\documentclass[a4paper]{amsart}

\usepackage{amssymb,amscd}
\usepackage{mathrsfs} 
\usepackage[all]{xy}

\usepackage[
backref,
pdfauthor={Han Wu},  
pdftitle={On genus one curves violating the local-global principle},
]{hyperref}
\usepackage[alphabetic,backrefs,lite]{amsrefs}  

\theoremstyle{plain}
\theoremstyle{definition}

\newtheorem{theorem}{Theorem}[subsection]
\newtheorem{thm}{Theorem}[subsubsection]
\newtheorem{lemma}[theorem]{Lemma}

\newtheorem{cor}[theorem]{Corollary}
\newtheorem{proposition}[theorem]{Proposition}
\newtheorem{prop}[theorem]{Proposition}

\theoremstyle{definition}

\newtheorem{assumption}[theorem]{Assumption}

\theoremstyle{remark}
\newtheorem{remark}[theorem]{Remark}

\renewcommand{\AA}{\mathbb{A}}

\newcommand{\QQ}{\mathbb{Q}}

\newcommand{\ZZ}{\mathbb{Z}}

\newcommand{\FF}{\mathbb{F}}

\newcommand{\PP}{\mathbb{P}}

\newcommand{\Ecal}{{\mathcal E}}
\newcommand{\Lcal}{{\mathcal L}}
\newcommand{\Ocal}{{\mathcal O}}

\newcommand{\mfr}{{\mathfrak m}}

\newcommand{\invap}{\inv_v(A(P_v))}

\DeclareMathOperator{\Gal}{Gal}

\DeclareMathOperator{\inv}{inv}
\DeclareMathOperator{\Proj}{Proj}
\DeclareMathOperator{\Sym}{Sym}

\usepackage[OT2,T1]{fontenc}
\DeclareSymbolFont{cyrletters}{OT2}{wncyr}{m}{n}
\DeclareMathSymbol{\Sha}{\mathalpha}{cyrletters}{"58}

\newcommand{\defi}[1]{\textsf{#1}} 

\newcommand{\BMo}{Brauer-Manin obstruction }
\newcommand{\Br}{\textup{Br}}

\makeatletter
\g@addto@macro\bfseries{\boldmath}  
\makeatother

\begin{document}
	
	\begin{title}
		{On genus one curves violating the local-global principle}  
	\end{title}
	\author{Han Wu}
	\address{University of Science and Technology of China,
		School of Mathematical Sciences,
		No.96, JinZhai Road, Baohe District, Hefei,
		Anhui, 230026. P.R.China.}
	\email{wuhan90@mail.ustc.edu.cn}
	\date{}
	\subjclass[2020]{Primary 14H45; Secondary 11G05, 14G12, 14G05.}
	\keywords{rational points, local-global principle., genus one curves, Brauer-Manin obstruction.}




	\begin{abstract} 
		For any number field not containing $\QQ(i),$  
		we give an explicit construction to prove that there exists an elliptic curve defined over this field such that its Shafarevich-Tate group is nontrivial.
	\end{abstract} 
	
	\maketitle

	
	\section{Introduction}
	
	\subsection{Background}
	Let $X$ be a proper algebraic variety defined over a number field $K.$ 
	Let $\Omega_K$ be the set of all nontrivial places of $K.$ 
	Let $K_v$ be the completion of $K$ at $v\in \Omega_K.$
	We say that $X$ violates the \defi{local-global principle} if $X(K_v)\neq\emptyset$ for all $v\in\Omega_K,$ whereas $X(K)=\emptyset.$ 
	
	The Hasse-Minkowski theorem states that an algebraic variety defined by a quadratic form satisfies the local-global principle. As a consequence, the local-global principle holds for every smooth, projective and geometrically connected curve of genus-0.  The first genus-1 curve violating the local-global principle was discovered by Lind \cite{Li40} and Reichardt \cite{Re42}.
	A simple one is Selmer's cubic curve defined over $\QQ$ by
	$3w_0^3+4w_1^3+5w_2^3=0$ in $\PP^2$ with homogeneous coordinates $(w_0:w_1:w_2),$ cf. \cite{Se51}.
	Poonen \cite{Po10a} proved that there exist curves over every global field violating the local-global
	principle. Clark \cite[Section 5 Conjecture 1]{Cl09} conjectured that genus-1 curve will be enough, i.e. for any given global field, there exists a genus-1 curve over this field violating the local-global principle. 
	
	Our goal is to prove that Clark's conjecture \cite[Section 5 Conjecture 1]{Cl09} holds for any number field not containing $\QQ(i).$ More exactly, we will prove the following theorem.
	
	\begin{thm}[Theorem \ref{thm: main theorem}]
		For any number field $K$ not containing $\QQ(i),$ there exists an elliptic curve $E$ defined over $K$ such that $\Sha (K,E)[2]\neq 0.$ Here  $\Sha (K,E)[2]$ is the $2$-torsion subgroup of the Shafarevich-Tate group of $E.$
	\end{thm} 
	
	The way to prove this theorem is to give an explicit construction of genus-1 curve violating the local-global principle. The curves that we consider in this paper, are 
	the smooth projective models of plane curves defined by the equation 
		\begin{equation}\label{equation}
		qy^2=x^4-p
	\end{equation}
    in $K[x,y]$ such that $p,~q$ are coprime integers in $K.$ These are genus-1 curves. We give smooth compactifications of these curves in Section \ref{Compactification}. We discuss in which case these curves have local points in Section \ref{section: existence of local points}, and give sufficient conditions about whether they don't exist $K$-rational points in Section \ref{section: no rational points}. In Section \ref{section: check the assumption}, we will prove that if the number field $K$ does not contain $\QQ(i),$ then we can choose parameters satisfying all conditions of Subsection \ref{subsection: assumption on a number field} so that the chosen curve violates the local-global principle. Then the Jacobian of the chosen curve will meet the needs of Theorem \ref{thm: main theorem}.
    
    Although we believe that for any number field, the curve given by the equation (\ref{equation}) can positively answer Clark's conjecture \cite[Section 5 Conjecture 1]{Cl09} completely, it seems unreachable now. Our paper in progress will answer Clark's conjecture \cite[Section 5 Conjecture 1]{Cl09} completely, but the proof is not given by an explicit construction.
    
	\section{Notation}
Given a number field $K,$ let $\Ocal_K$ be the ring of its integers, and let $\Omega_K$ be the set of all its nontrivial places. Let $\infty_K\subset \Omega_K$ be the subset of all archimedean places, and let $2_K\subset \Omega_K$ be the subset of all $2$-adic places. Let $\infty_K^r\subset \infty_K$ be the subset of all real places, and let $\Omega_K^f=\Omega_K\backslash \infty_K$ be the set of all finite places of $K.$ Let $K_v$ be the completion of $K$ at  $v\in \Omega_K.$ 
 For $v\in \Omega_K^f,$  let $\FF_v$ be its residue field. 
We say that an element is a \defi{prime element}, if the ideal generated by this element is a prime ideal.

	\section{Compactification}\label{Compactification}
	In this section, we construct the smooth projective models of plane curves defined by the equation (\ref{equation}).
	
\subsection{Zero loci in projective space bundles}\label{subsection: compactification}
Let $B$ be a smooth, projective, and geometrically connected variety over a number field $K.$ Let $\Lcal$ be a line bundle on $B,$ and we assume the set of global sections $\Gamma(B,\Lcal^{\otimes2})\neq 0.$ Let $\Ecal=\Ocal_B\oplus \Lcal$ be the rank-$2$ vector bundle on $B.$ Let $a\in K^\times$ be a constant, and let $s\in \Gamma(B,\Lcal^{\otimes2})$ be a nonzero global section. The zero locus of $(a,-s)\in\Gamma(B,\Ocal_B\oplus\Lcal^{\otimes2})\subset\Gamma(B,\Sym^2\Ecal)$ in the projective space bundle $\Proj(\Ecal)$ is  a projective scheme, denoted by $X$ with the natural projection $\alpha\colon X\to B.$ We have the following lemma to ensure the smoothness of $X.$

\begin{lemma}\label{lemma smoothness criterion}
	Given a number field $K,$ we use the notation as in Subsection \ref{subsection: compactification}. The locus defined by $s=0$ in $B$ is smooth, if and only if $X$ is smooth.
\end{lemma}

\begin{proof}
	By the Jacobian criterion and local computation, this lemma follows.
\end{proof}

\begin{remark}\label{remark: smooth compactification}
	Consider the case that $B=\PP^1$ and $\Lcal=\Ocal(n_0)$ for some fixed positive integer $n_0.$ Let  $s\in \Gamma(\PP^1,\Ocal(2n_0))$ be a global section, and let $s(x)\in K[x]$ be a dehomogenization of $s.$ By Lemma \ref{lemma smoothness criterion}, in order to ensure the smoothness of $X,$ we assume that the polynomial $s(x)$ is separable of degree $2n_0.$ Then $X$ is a smooth, projective, and geometrically connected curve, and it is the smooth compactification of the affine curve given by the equation $ay^2=s(x)\in \AA^2$ with affine coordinates $(x,y).$ 
\end{remark}

Next, we consider the plane curve defined by the equation (\ref{equation}). By Remark \ref{remark: smooth compactification}, we choose the smooth projective model as in Subsection \ref{subsection: compactification}.

\begin{cor}\label{cor: genus =1}
	Given a number field $K,$ let $C$ be 
	the smooth projective model of the plane curve defined by the equation (\ref{equation}).
	 Then the genus of $C$ is one. 
\end{cor}

\begin{proof}
	Since the polynomial $x^4-p$ is separable over $K,$ by Lemma \ref{lemma smoothness criterion} and Remark \ref{remark: smooth compactification}, we consider its smooth compactification as in \ref{subsection: compactification}, and let $\alpha\colon C\to \PP^1$ be the natural projection. By Hurwitz's Theorem \cite[Chapter IV. Corollary 2.4]{Ha97}, the genus of $C$ is one.
\end{proof}
\section{Existence of local points}\label{section: existence of local points}

	In this section, we discuss in which case the curve defined by the equation (\ref{equation}) has local points. By Lang-Weil estimate \cite{LW54}, a given curve over a number field $K$ has local points for almost all\footnote{The phrase "almost all" means "all but finite".} places of $K.$ 
	 The following proposition tells all possible places where it doesn't have local points. It is essential to judge whether the curve has local points.

   	\begin{proposition}\label{prop: exist local points}	Given a number field $K,$ let $C^0$ be the plane curve defined by the equation $qy^2=x^4-p$ in $K[x,y]$ such that $p,~q$ are coprime integers in $K.$ Let $S= \infty_K^r \cup 2_K\cup \{v\in \Omega_K^f | v(pq)\neq 0\}$ be a finite set.   		
   		 Then $C^0(K_v)\neq \emptyset$ for all $v\in \Omega_K\backslash S.$
\end{proposition}

    \begin{proof}
    		Let $C$ be the smooth compactification of $C^0.$ By the implicit function theorem, we only need to check that $C(K_v)\neq \emptyset.$
    		 Suppose that  $v\in \Omega_K\backslash S,$ then $v$ is an odd place. If $q\in K_v^{\times 2},$ then the curve $C^0$ admits a $K_v$-point with $v(x)<0.$ Next, we consider the case $q\notin K_v^2.$ The curve $C$ has good reduction over $K_v.$ Let $\tilde{C}$ be its reduction. Since any point in $\tilde{C}(\FF_v)$ can be lifted to a point in $C(K_v),$ we only need to prove that $\tilde{C}(\FF_v)\neq \emptyset.$ Let $C'$ be the smooth compactification of the curve defined by 
    		 $y^2=x^4-p$ in $K[x,y],$ and $\tilde{C}'$ be its reduction. By Remark \ref{remark: smooth compactification} and Corollary \ref{cor: genus =1}, the curve  $\tilde{C}'$ is an elliptic curve over $\FF_v.$ By Hasse's bound for elliptic over finite field \cite[Chapter V, Theorem 1.1]{Si09}, we have $\sharp \tilde{C}'(\FF_v)\leq 1+\sharp\FF_v+2\sqrt{\sharp\FF_v}.$ So, we get $\sharp \tilde{C}'(\FF_v)<2(1+\sharp\FF_v).$ If the polynomial $x^4-p$ has a $\FF_v$-root, then $\tilde{C}(\FF_v)\neq \emptyset.$ Otherwise, we consider the map $\alpha\colon \tilde{C}'\to \PP^1$ given by $(x,y)\mapsto x,$ the same as in Subsection \ref{subsection: compactification}. Then $\sharp \alpha(\tilde{C}'(\FF_v))=\sharp\tilde{C}'(\FF_v)/2<1+\sharp\FF_v.$ So $\PP^1(\FF_v)\backslash \alpha(\tilde{C}'(\FF_v))\neq \emptyset.$ In this case, there exists an element $\overline{x_0}\in \FF_v$ such that $\overline{x_0}^4-p\notin\FF_v^2.$ For $q\notin K_v^2,$ the equation $qy^2=x^4-p$ with affine coordinates $(x,y)$ has a $\FF_v$-solution with $x=\overline{x_0}.$ So in each case, we have
    		 $\tilde{C}(\FF_v)\neq \emptyset.$ So $C^0(K_v)\neq \emptyset.$
    	    \end{proof}
        	 
    We put some additional conditions for the parameters $p,~q$ so that the curve in Proposition \ref{prop: exist local points} has local points for all places.

   \begin{cor}\label{Cor: exist local points}
   	Given a number field $K,$ let $C^0$ be the plane curve defined by the equation $qy^2=x^4-p$ in $K[x,y]$ such that $p,~q$ are coprime integers in $K.$ Let $S_1=\infty_K^r \cup 2_K\cup \{v\in \Omega_K^f | v(q)\neq 0\}.$ Additionally, we assume that $p$ is a prime element, and $p\in K_v^{\times 4}$ for all $v\in S_1.$ Then $C^0(K_v)\neq \emptyset$ for all $v\in \Omega_K.$
   \end{cor}		 
    		   \begin{proof}
    		   	By Proposition \ref{prop: exist local points}, we only need to check the case that $v\in S_1\cup\{v_{p}\}.$
    		   	
    		   	Suppose that $v\in  S_1.$ By assumption that $p\in K_v^{\times 4},$ then the equation $qy^2=x^4-p$ with affine coordinates $(x,y)$ has a $K_v$-solution with $y=0.$\\
    		   	Suppose that $v=v_p,$ then $v$ is an odd place. By the product formula $\prod_{v\in \Omega_K}(p,q)_v= 1,$ we have $q\in K_v^{\times 2}.$ By Hensel's lemma, the curve $C^0$ admits a $K_v$-point with $v(x)<0.$
    		   	
    		   	So $C^0(K_v)\neq \emptyset$ for all $v\in \Omega_K.$
    		\end{proof}

    \section{The local-global principle for genus one curves}\label{section: no rational points}
    \subsection{The \BMo}
        	Let $C^0$ be a curve defined over a number field $K$ as in Proposition \ref{prop: exist local points}, and let $C$ be its smooth compactification. Let $K(C^0)$ be the function field of $C^0.$ According to \cite[Cor. 1.8]{Gr68II}, we have an inclusion $\Br(C)\hookrightarrow\Br(K(C^0)).$     	
    Next, we consider the quaternion algebra class $(y,p)\in \Br(K(C^0))[2].$ It is proved in \cite[Example 6.3.3]{CTS21} that this quaternion algebra class indeed belongs to $\Br(C)[2].$ They proved it in a purely algebraic way. To fit into the context of this paper, we will give another proof by explicit calculations of local invariants. 
    
    \begin{lemma}\label{lemma: Brauer group extension}
    	Given a number field $K,$ let $C^0$ be the plane curve defined by the equation $qy^2=x^4-p$ in $K[x,y]$ such that $p,~q$ are coprime integers in $K.$ Then the quaternion algebra class $A=(y,p)\in \Br(K(C^0))[2]$ belongs to $\Br(C)[2].$ 
    \end{lemma}
    
    \begin{proof}
    	By Harari's formal lemma \cite[Th\'eor\`eme 2.1.1]{Ha94}, it will be sufficient to prove that for almost all $v\in \Omega_K,$ the local invariant $\invap=0$ for all $P_v\in C^0(K_v)$ with $y(P_v)\neq 0.$
    	
    	Let $S= \infty_K^r \cup 2_K\cup \{v\in \Omega_K^f | v(pq)\neq 0\}$ be a finite set.   Suppose that $v\in \Omega_K\backslash S,$ then $v$ is an odd place. Take an arbitrary $P_v\in C^0(K_v)$ with $y(P_v)\neq 0.$ If $\invap=1/2,$ then $(y,p)_v=-1$ at $P_v.$ By Chevalley-Warning theorem  \cite[Chapter I \S 2, Corollary 2]{Se73}, the valuation $v(y)$ is odd at $P_v,$ which implies $v(y)>0.$ So the polynomial $x^4-p$ has a $\FF_v$-root. By Hensel's lemma, we have $p\in K_v^{\times 2},$ which contradicts $(y,p)_v=-1.$ So $\invap=0.$
    \end{proof}
	
	We make the following assumption on the base field and the parameters of the equation (\ref{equation})  so that the curve given by this equation violates the local-global principle.

\subsection{Assumption on a number field}\label{subsection: assumption on a number field}Let $K$ be a number field. For the ideal class group of $K$ is finite, we take a positive integer $N$ such that $\Ocal_K[1/N]$ is a principle ideal domain. Let $S_2=\infty_K^r\cup 2_K\cup \{v\in \Omega_K^f | v(N)\neq 0\}$ be a finite set.  We make the following assumption on the number field $K.$
\begin{assumption}\label{assumption: assumption on a number field} For a number field $K,$ we assume that there exist a pair of coprime integers $(p,~q)$ and an odd place $v_0\in \Omega_K$ satisfying the following assumption.	
	\begin{enumerate}
		\item\label{assumption 1} $p$ is a prime element,
		\item\label{assumption 2} $p\in K_v^{\times 4}$  for all $v\in S_2,$ 
		\item\label{assumption 3} $q\in K_p^{\times 2}\backslash K_p^{\times 4},$
		\item\label{assumption 4}  $\sharp \FF_{v_0}\equiv 3\mod 4,$ 
		\item\label{assumption 5}  $v_0(q)=1,$	
		\item\label{assumption 6} $\{v\in \Omega_K^f | v(q)\neq 0\}\backslash \{v_0\} \subset \{v\in \Omega_K^f | v(N)\neq 0\}.$
	\end{enumerate}
\end{assumption}

\subsection{The local-global principle}
The following proposition states that the curve given by the equation (\ref{equation}) with parameters satisfy Assumption \ref{assumption: assumption on a number field} has local points.
\begin{prop}\label{proposition: exist of local point}
	For a number field $K,$ we assume that there exist a pair of coprime integers $(p,~q)$ and an odd place $v_0\in \Omega_K$ satisfying Assumption \ref{assumption: assumption on a number field}. let $C^0$ be the plane curve defined by the equation $qy^2=x^4-p$ in $K[x,y].$ Then $C^0(K_v)\neq \emptyset$ for all $v\in \Omega_K.$
\end{prop} 

\begin{proof}
	We will check that the parameters satisfy all conditions of Corollary \ref{Cor: exist local points}.  Let $S_1=\infty_K^r \cup 2_K\cup \{v\in \Omega_K^f | v(q)\neq 0\}$ as in Corollary \ref{Cor: exist local points}, and let $S_2=\infty_K^r\cup 2_K\cup \{v\in \Omega_K^f | v(N)\neq 0\}$ as in Subsection \ref{assumption: assumption on a number field}, then $S_1\backslash \{v_0\}\subset S_2.$ By the choice of $p$ and Corollary \ref{Cor: exist local points}, we only need to prove that $p\in K_{v_0}^{\times 4}.$ By the product formula $\prod_{v\in \Omega_K}(p,q)_v=1$ and $q\in K_p^{\times 2},$ we have $(p,q)_{v_0}=1.$ By Assumption (\ref{assumption 5}), the reduction $\bar{p}\in \FF_{v_0}^{\times 2}.$ By Assumption (\ref{assumption 4}) that $\sharp \FF_{v_0}\equiv 3\mod 4,$ we have $\FF_{v_0}^{\times 2}=\FF_{v_0}^{\times 4},$ so $\bar{p}\in \FF_q^{\times 4}.$ By Hensel's lemma, we have $p\in K_{v_0}^{\times 4}.$ 
\end{proof}

By Lemma \ref{lemma: Brauer group extension}, the quaternion algebra class $A=(y,p)\in \Br(K(C^0))[2]$ belongs to $\Br(C)[2].$ 
For the curve given in Proposition \ref{proposition: exist of local point}, we calculate the local invariant $\invap$ for all $P_v\in C(K_v)$ and all $v\in \Omega_K.$ 

\begin{lemma}\label{lemma: violating local-global principle}
		For a number field $K,$ we assume that there exist a pair of coprime integers $(p,~q)$ and an odd place $v_0\in \Omega_K$ satisfying Assumption \ref{assumption: assumption on a number field}. Let $C$ be the smooth projective model of the plane curve defined by the equation $qy^2=x^4-p$ in $K[x,y].$ Let $A=(y,p)\in \Br(C)[2].$ Then, for any $v\in \Omega_K,$ and any $P_v\in C(K_v),$ 
			\begin{equation*}
				\inv_v(A(P_v))=\begin{cases}
					0& if\quad v\neq v_p,\\
					1/2 & if \quad v= v_p.
				\end{cases}
			\end{equation*}
\end{lemma}

\begin{proof}
	For the evaluation of $A$ on $C(K_v)$ is locally constant, the implicit function theorem implies that we only need to calculate $\inv_v(A(P_v))$ for all $P_v\in C^0(K_v)$ with $y(P_v)\neq 0.$
	
		Let $S= \infty_K^r \cup 2_K\cup \{v\in \Omega_K^f | v(pq)\neq 0\}$ be a finite set. As in the proof of Lemma \ref{lemma: Brauer group extension}, for any $v\in \Omega_K\backslash S,$ and any $P_v\in C^0(K_v)$ with $y(P_v)\neq 0,$ we have $\inv_v(A(P_v))=0.$ \\
			Suppose that $v\in S\backslash \{v_0,v_p\}.$ Let $S_2=\infty_K^r\cup 2_K\cup \{v\in \Omega_K^f | v(N)\neq 0\}$ as in Subsection \ref{assumption: assumption on a number field}, then $v\in S_2.$ By Assumption (\ref{assumption 2}), we have $\inv_v(A(P_v))=0$ for all  $P_v\in C^0(K_v)$  with $y(P_v)\neq 0.$\\
		Suppose that  $v=v_0.$ By the same argument as in the proof of Proposition \ref{proposition: exist of local point}, we have  $p\in K_{v}^{\times 4},$ which implies  $\inv_v(A(P_v))=0$ for all  $P_v\in C^0(K_v)$  with $y(P_v)\neq 0.$
		
		Suppose that  $v=v_p.$ Take an arbitrary $P_v\in C^0(K_v).$   By comparing the valuation of the equation (\ref{equation}), we have $v(x)\leq 0$ and $x^4-p\in K_v^{\times 4}$ at $P_v.$ By Assumption (\ref{assumption 3}), we have $q\in K_v^{\times 2}\backslash K_v^{\times 4}.$ Then $y\notin K_v^{2},$ so $\inv_v(A(P_v))=1/2$ at $P_v.$	
\end{proof}
	Combining Lemma \ref{lemma: violating local-global principle} with the global reciprocity law, we have
the following proposition that the curve given by the equation (\ref{equation}) with parameters satisfying Assumption \ref{assumption: assumption on a number field} has no $K$-rational point.

\begin{prop}\label{proposition: no rational point}
	For a number field $K,$ we assume that there exist a pair of coprime integers $(p,~q)$ and an odd place $v_0\in \Omega_K$ 
	satisfying Assumption \ref{assumption: assumption on a number field}. Let $C$ be the smooth projective model of the plane curve defined by the equation $qy^2=x^4-p$ in $K[x,y].$ Then $C(K)= \emptyset.$
\end{prop} 

\begin{proof}
	If  the curve $C$ has a $K$-rational point $P,$ then by the global reciprocity law, the sum $\sum_{v\in \Omega_K}\inv_v(A(P))=0$ in $\QQ/\ZZ.$ But by Lemma \ref{lemma: violating local-global principle}, 	the sum is $1/2$, which is nonzero in $\QQ/\ZZ.$
\end{proof}

Applying Proposition \ref{proposition: no rational point}, the Jacobian of $C$ will have the following proposition.
	\begin{prop}\label{prop: nontrivial Sha}
	For a number field $K,$ we assume that there exist a pair of coprime integers $(p,~q)$ and an odd place $v_0\in \Omega_K$ satisfying Assumption \ref{assumption: assumption on a number field}. Then there exists an elliptic curve $E$ defined over $K$ such that $\Sha (K,E)[2]\neq 0.$ 
\end{prop} 

\begin{proof}
	Let $C$ be the smooth projective model of the plane curve defined by the equation $qy^2=x^4-p$ in $K[x,y].$ Let $E$ be the Jacobian of $C.$ By Corollary \ref{Cor: exist local points}, the curve $E$ is an elliptic curve. Consider the class $[C]\in H^1(K, E).$ By Proposition \ref{proposition: exist of local point}, the set $C(K_v)\neq \emptyset$ for every $v\in \Omega_K,$ which implies $[C]\in \Sha (K,E).$ By Proposition \ref{proposition: no rational point}, the set $C(K)=\emptyset,$ which implies $[C]\neq 0.$ For $C(K(\sqrt{q}))\neq\emptyset,$  the standard restriction-corestriction argument implies $2[C]=0,$ so $[C]\in \Sha (K,E)[2]$ is a nonzero element. 
\end{proof}

    \section{Assumption \ref{assumption: assumption on a number field} holds for number fields not containing $\QQ(i)$}\label{section: check the assumption}
    
    Let $K$ be a number field not containing $\QQ(i).$  In this section, we will find a pair of coprime integers $(p,~q)$ and an odd place $v_0\in \Omega_K$ satisfying Assumption \ref{assumption: assumption on a number field}.

    The following lemma is a consequence of the \v{C}ebotarev density theorem.
    
       		\begin{lemma}\label{lemma mod 4=3 prime element}    
       			Let $K$ be a number field not containing $\QQ(i).$ Let $P_K$ be the set of prime numbers $p$ such that
       			\begin{itemize}
       				\item  $p$ splits over $K,$ i.e. there exists a place $v\in \Omega_K$ such that $K_v\cong \QQ_p.$
       				\item $p\equiv 3\mod 4.$ 
       			\end{itemize}       			
       			 Then the set $P_K$ is a infinite set.
    \end{lemma}

    \begin{proof}
	 Let $\overline{K}$ be an algebraic closure field of $K.$ Let $L$ be the Galois closure of $K/\QQ$ in $\overline{K}.$ For $\QQ(i)\nsubseteq K,$ the extension $L(i)/K$ is a nontrivial Galois extension. Take an element $\sigma\in \Gal(L(i)/K)$ such that the restriction of $\sigma$ on $K(i)$ is nontrivial. Let $P_{L(i)}$ be the set of all prime numbers $p$ such that there exists a place $v|p$ in $\Omega_{L(i)}$ satisfying that the subgroup group $\Gal(L(i)_{v}/\QQ_p)\subset \Gal(L(i)/\QQ)$ is generated by $\sigma.$
    	 Similar to the proof of \cite[Chapt. VII Theorem 13.4]{Ne99}, the set $P_{L(i)}$ has positive density. Take an arbitrary prime number $p_0\in P_{L(i)}$  unramified over $L(i).$ Then there exists a place $v_0|p_0$ in $\Omega_{L(i)}$ such that $\Gal(L(i)_{v_0}/\QQ_{p_0})$ is generated by $\sigma.$ Let $w_0',~w_0$ be the restriction of $v_0$ on $K(i)$ and $K$ respectively. Then $K_{w_0}\cong \QQ_{p_0},$ and $K(i)_{w_0'}/K_{w_0}$ is a nontrivial unramified extension. So $p_0$ splits over $K,$ and $p_0\equiv 3\mod 4.$ So the set $P_K$ is a infinite set.
    \end{proof}

        \begin{remark}
    	One can calculate that the density $d(P_K)\geq \frac{1}{2[K:\QQ]}.$ Here $[K:\QQ]$ is the degree of the extension $K/\QQ.$
    \end{remark}

Given a number field $K$ not containing $\QQ(i),$ let $N$ and $S_2=\infty_K^r\cup 2_K\cup \{v\in \Omega_K^f | v(N)\neq 0\}$ be as in Subsection \ref{subsection: assumption on a number field}. We will choose $(p,~q)$ and an odd place $v_0\in \Omega_K$ satisfying Assumption \ref{assumption: assumption on a number field}.
\begin{prop}\label{prop: existence of parameters}
	Let $K$ be a number field not containing $\QQ(i).$
	Then there exist a pair of coprime integers $(p,~q)$ and an odd place $v_0\in \Omega_K$ satisfying Assumption \ref{assumption: assumption on a number field}.
\end{prop}

\begin{proof}
		Let $\mfr_{\infty}$ be the product of all places in $\infty_K^r,$ and let $\mfr=16N^2 \mfr_{\infty}$ be a modulus of $K.$ Let $K_\mfr$ be the ray class field of modulus $\mfr,$ then $\QQ(i)\subset  K_\mfr.$
	Let $I_\mfr$ be the group of fractional ideals that are prime to  $16N^2.$ Let $P_\mfr\subset I_\mfr$ be the subgroup of principal ideals generated by some $a\in K^\times$ with $a\equiv 1 \mod 16N^2$ and $\tau_v(a)>0$ for all $v\in \infty_K^r.$ Then by Artin reciprocity law
	(cf. \cite[Chapter VI Theorem 7.1 and Corollary 7.2]{Ne99}), the classical Artin homomorphism $\theta$ gives an exact sequence:
	\begin{equation}\label{artin exact sequence}
		0\to P_\mfr\hookrightarrow I_\mfr \stackrel{\theta}\to \Gal(K_\mfr/K)\to 0.
	\end{equation}
	By Lemma \ref{lemma mod 4=3 prime element}, we take a place $v_0\in \Omega_K$ such that $\sharp \FF_{v_0}\equiv 3 \mod 4,$ and $v_0$ unramified over $K_\mfr.$ For $\Ocal_K[1/N]$ is a principle ideal domain, we can take $q\in \Ocal_K$ satisfying Assumption (\ref{assumption 5})  and Assumption (\ref{assumption 6}). For $v_0$ is unramified over $K_\mfr$ and $v_0(q)=1,$ the Galois group $\Gal(K_\mfr(\sqrt[4]{q})/K_\mfr)\cong \ZZ/4\ZZ$ denoted a generator by $\sigma.$ 
	By \v{C}ebotarev density theorem \cite[Chapt. VII Theorem 13.4]{Ne99}, there exist a place $v_1'$ in $\Omega_{K_\mfr(\sqrt[4]{q})}$ and its restriction $v_1\in \Omega_K$ such that $\Gal(K_\mfr(\sqrt[4]{q})_{v_1'}/K_{v_1})$ is generated by $\sigma^2.$ Then $v_1$ splits completely in $K_\mfr.$ By \cite[Chapt. VI Corollary 7.4]{Ne99} and the exact sequence (\ref{artin exact sequence}), the place $v_1$ is associated to a prime element $p\in P_\mfr.$ So Assumption (\ref{assumption 1}) and Assumption (\ref{assumption 2}) hold. By the choice of $v_1,$ the assumption $q\in K_p^{\times 2}\backslash K_p^{\times 4}$ holds.
\end{proof}

	
	Combining Proposition \ref{prop: nontrivial Sha} and Proposition \ref{prop: existence of parameters}, we get our theorem.
	
	\begin{thm}\label{thm: main theorem}
	For any number field $K$ not containing $\QQ(i),$ there exists an elliptic curve $E$ defined over $K$ such that $\Sha (K,E)[2]\neq 0.$ 
\end{thm}

	\begin{footnotesize}
		\noindent\textbf{Acknowledgements.} The author would like to thank D.S. Wei, Y. Xu and C. Lv for many fruitful discussions.  The author is grateful to B. Poonen and anonymous referees for their valuable suggestions. The author was partially supported by NSFC Grant No. 12071448.
	\end{footnotesize}


\end{document}